\documentclass[12pt]{amsart}
\usepackage{amsmath,latexsym,amsfonts,amssymb,amsthm}
\usepackage{geometry}
\usepackage[usenames,dvipsnames]{xcolor}
\usepackage[colorlinks=true,citecolor=Green,linkcolor=Red]{hyperref}
\usepackage[nameinlink,capitalize]{cleveref}
\usepackage{mathrsfs}
\usepackage{graphicx,color}
\usepackage[alphabetic]{amsrefs}
\usepackage[all,cmtip]{xy}
\usepackage{bbm, stmaryrd}
\usepackage{tabu}
\usepackage{enumitem}
\usepackage{tikz}
\usepackage{tikz-cd}

\newtheorem{prop}{Proposition}[section]
\newtheorem{thm}[prop]{Theorem}

\newtheorem{conj}[prop]{Conjecture}
\newtheorem{lem}[prop]{Lemma}

\theoremstyle{definition}

\newtheorem{defn}[prop]{Definition}

\newtheorem{rem}[prop]{\it Remark}
\newtheorem{emp}[prop]{Example}

\newtheorem*{mthm}{Main Theorem}
\newtheorem*{claim*}{Claim}

\DeclareMathOperator{\Aut}{Aut}
\DeclareMathOperator{\Iso}{Iso}

\DeclareMathOperator{\Ricci}{Ric}

\DeclareMathOperator{\ddbar}{\partial\overline{\partial}}

\newcommand{\bC}{\mathbb{C}}
\newcommand{\bR}{\mathbb{R}}

\newcommand{\bZ}{\mathbb{Z}}

\numberwithin{equation}{section}

\title[]{finite generation of the ring of holomorphic functions with polynomial growth on the K\"{a}hler-Ricci shrinker}

\author{Jiangtao Li}
\address{Department of Mathematics, University of California San Diego, 9500 Gilman Drive, La Jolla, CA}
\email{jil320@ucsd.edu}

\begin{document}

\maketitle

\begin{abstract}
    Let $(X,g,J,f)$ be a non-compact gradient shrinking K\"{a}hler-Ricci soliton. We prove that if the scalar curvature of $X$ satisfies a mild assumption, then $\mathcal{O}_P(X)$, the ring of holomorphic functions with polynomial growth on $X$, is finitely generated. This gives a partial confirmation to a conjecture of Munteanu and Wang (cf.\cite{MW14}). 
\end{abstract}

\section{Introduction}

The classical Liouville theorem asserts that any holomorphic function on $\bC^n$ with polynomial growth of order at most $d$ is indeed a polynomial function of degree at most $d$. In particular, this implies that the ring of holomorphic functions with polynomial growth is finitely generated. As part of his program of the uniformization theorems for complete K\"{a}hler manifolds, Yau made the following conjecture, which generalizes the classical theorem (cf.\cite{SY94}): 
\begin{conj}\label{conj:YausConjecture}
    Let $(X^n,g,J)$ be a complex $n$ dimensional complete K\"{a}hler manifold with nonnegative bisectional curvature and $\mathcal{O}(X)$ be the ring of holomorphic functions on $X$. Define the space of holomorphic functions with polynomial growth of order at most $d$ by
    \begin{equation*}
        \mathcal{O}_d(X)=\{f\in\mathcal{O}(X):\exists C>0, |f(x)|\leq C(1+r(x))^d\}.
    \end{equation*}
    Then
    \begin{enumerate}
        \item $\dim\mathcal{O}_d(X)\leq \dim\mathcal{O}_d(\bC^n)$; 
        \item the ring of all holomorphic functions with polynomial growth 
        \begin{equation*}
            \mathcal{O}_P(X)=\cup_{d\ge 0}\mathcal{O}_d(X)
        \end{equation*}
        is finitely generated.
    \end{enumerate}
\end{conj}

The first part of the above conjecture was confirmed under the additional assumption that $X$ has maximal volume growth by Ni (cf.\cite{Ni04}) using a monotonicity formula along the heat flow on K\"{a}hler manifolds. By modifying Ni's method, the volume growth condition was removed in Chen-Fu-Le-Zhu (cf.\cite{CFYZ06}). More recently, by deploying a different approach, Liu (cf.\cite{Liu16b}) showed that it is true under the weaker assumption that $X$ has nonnegative holomorphic sectional curvature.

In contrast, the second part of the conjecture, known as finite generation conjecture, remained open until Liu gave a proof (cf.\cite{Liu16a}), among other things.

Similar Liouville type problems for K\"{a}hler-Ricci solitons were proposed and studied by Munteanu-Wang (cf.\cite{MW14}). They proved that $\dim(\mathcal{O}_d(X))<\infty$ for any shrinking gradient K\"{a}hler-Ricci soliton $(X,g,J,f)$ and get an upper bound of $\dim(\mathcal{O}_d(X))$ as a polynomial of $d$, which depends on the soliton structure. However, their method doesn't imply the finite generation of $\mathcal{O}_P(X)$. 

In their recent work (cf.\cite{SZ25}), Sun-Zhang introduced the notion of polarized Fano fibrations to study the Yau-Tian-Donaldson type conjectures. Among other things, they constructed a polarized Fano fibration structure on any K\"{a}hler-Ricci shrinker. With help of this structure theorem, we are able to resolve the finite generation problem on K\"{a}hler-Ricci shrinkers, under a mild assumption on the growth rate of scalar curvature. The main theorem of this work is:

\begin{mthm}\label{thm:MainTheorem}
    Let $(X,g,J,f)$ be a non-compact gradient K\"{a}hler-Ricci shrinking soliton. Let $R$ be the scalar curvature and assume that
    \begin{equation}\label{eq:ScalarCurvatureBound}
        \limsup_{r(x)\to\infty}\frac{R(x)}{(1+r(x))^2}=\gamma<1.
    \end{equation}
    where $r(x)$ is the distance from $x$ to some fixed point $p\in X$. Then the ring of holomorphic functions on $X$ with polynomial growth $\mathcal{O}_P(X)$ is isomorphic to the ring of regular functions on an affine variety. In particular, it is finitely generated.
\end{mthm}

\begin{rem}
    In general, it is known that $\gamma\leq 1$ for any K\"{a}hler-Ricci shrinker (see \ref{prop:FundamentalResultsOnShrinkers}). Therefore, \eqref{eq:ScalarCurvatureBound} is a mild assmuption. Indeed, all the known examples of K\"{a}hler-Ricci shrinkers have bounded scalar curvature. In complex dimension 2, this boundedness was proved by Li-Wang (cf.\cite{LW25}). It remains open in higher dimensions.
\end{rem}

\subsection*{Outline of the paper} In section 2, we review some fundamental results on gradient K\"{a}hler-Ricci solitons and derive a lower gradient estimate for the potential function from the assumption on the scalar curvature. In section 3, we review the notion of polarized Fano fibrations and its relation with gradient shrinking K\"{a}hler-Ricci solitons. This was established in \cite{SZ25}. In the last section, by combining the estimate of section 2 and the polarized Fano fibration structure, we build connections between the distance on the K\"{a}hler-Ricci soliton and the distance on the affine base in the fibration. This allows us to show that the holomorphic functions on $X$ with polynomial growth are exactly pullbacks of those on the affine base of the fibration. The main theorem then follows from a classical result on the holomorphic functions with polynomial growth on affine varieties.

\section{K\"{a}hler-Ricci shrinkers}

\begin{defn}[K\"{a}hler-Ricci shrinker]
    Let $(X,g,J)$ be a complete $n$ dimensional K\"{a}hler manifold and $f\in C^\infty(X,\bR)$ a smooth real valued function on $X$. Suppose that $\omega$ is the associated K\"{a}hler form and $\Ricci(\omega)$ the Ricci form of $\omega$. $(X,g,J,f)$ is a \textit{gradient shrinking K\"{a}hler-Ricci soliton}, or \textit{K\"{a}hler-Ricci shrinker}, if the following equation holds 
    \begin{equation}\label{eq:KahlerRicciShrinker}
        \Ricci(\omega)+\sqrt{-1}\ddbar f=\omega.
    \end{equation}
    $f$ is called the \textit{potential function} of the K\"{a}hler-Ricci shrinker and $\xi=J\nabla f$ is called the \textit{soliton vector field} of the K\"{a}hler-Ricci shrinker.
\end{defn}

K\"{a}hler-Ricci shrinkers arise naturally as singularity models in the context of K\"{a}hler-Ricci flows and have been studied extensively in both compact and non-compact cases (cf.\cite{TZ00},\cite{WZ04},\cite{CDS24},\cite{CCD22},\cite{BCCD24},\cite{LW25}, etc.). 

We recall some of the basic results on K\"{a}hler-Ricci shrinkers:
\begin{prop}\label{prop:FundamentalResultsOnShrinkers}
    For any K\"{a}hler-Ricci shrinker $(X,g,J,f)$, the following hold:
    \begin{enumerate}
        \item $R+2|\nabla f|^2-2f$ is a constant. Therefore, we may normalize $f$ so that
        \begin{equation}\label{eq:NormalizationOnf}
            R+2|\nabla f|^2-2f=0.
        \end{equation}
        \item $R\ge 0$;
        \item $f$ is proper and bounded below. Let $p$ be a point where $f$ attains its minimum and $r(x)$ be the distance between $x$ and $p$. Suppose $f$ satisfies \ref{eq:NormalizationOnf}, then 
        the following estimates of the potential function $f$ hold:
        \begin{equation}\label{eq:BoundOnfByDistance}
            \frac{1}{2}(r(x)-c_1)_+^2\leq f(x) \leq \frac{1}{2}(r(x)+c_2)^2.
        \end{equation}
        Here $g_+=\max\{g,0\}$ and $c_1,c_2$ are constants depending only on the dimension $n$.
    \end{enumerate}
\end{prop}

\begin{proof}
    (1) follows by taking the trace and divergence of \eqref{eq:KahlerRicciShrinker}. For (2), see \cite{Cho23}*{Theorem 2.14}. For (3), see \cite{CZ10} or \cite{Cho23}*{Corollary 2.15, Theorem 4.3}.
\end{proof}

Using the above proposition, a gradient estimate of $f$ under the curvature assumption \eqref{eq:ScalarCurvatureBound} can be derived:
\begin{lem}\label{lem:GradientEstimate}
    Under the assumption \eqref{eq:ScalarCurvatureBound}, for any $0<\beta<1-\gamma$, there exists $r_\beta>0$ depending on $\beta$, such that for any $x\notin B_{r_\beta}(p)$, 
    \begin{equation}
        |\nabla f(x)|^2\ge \beta f(x).
    \end{equation}
\end{lem}
\begin{proof}
    Since $\beta<1-\gamma$, we may choose $\delta_1>0$ such that $(1-\beta)(1-\delta_1)^2>\gamma$. Suppose that when $r(x)>r_1$, there holds
    \begin{equation}\label{eq:GradientEstimateEq1}
        r(x)-c_1>(1-\delta_1)r(x).
    \end{equation}
    Let $\delta_2>0$ and $\gamma<1-\delta_2<(1-\beta)(1-\delta_1)^2$. By the assumption on the scalar curvature \eqref{eq:ScalarCurvatureBound}, there is $r_2>0$ and when $r(x)>r_2$, 
    \begin{equation}\label{eq:GradientEstimateEq2}
        \frac{R(x)}{(1+r(x))^2}\leq 1-\delta_2.
    \end{equation}
    Let $r_\beta=\max\{r_1,r_2\}$. It follows from \eqref{eq:NormalizationOnf}, \eqref{eq:BoundOnfByDistance}, \eqref{eq:GradientEstimateEq1} and \eqref{eq:GradientEstimateEq2} that if $r(x)\ge r_\beta$,
    \begin{align*}
        & |\nabla f(x)|^2=f(x)-\frac{R(x)}{2}=f(x)\left(1-\frac{R(x)}{2f(x)}\right)\ge f(x)\left(1-\frac{R(x)}{(r(x)-c_1)^2}\right) \\
        & > f(x)\left(1-\frac{1-\delta_2}{(1-\delta_1)^2}\right) >\beta f(x).
    \end{align*}
    The last inequality makes use of the relation $1-\delta_2<(1-\beta)(1-\delta_1)^2$.
\end{proof}

All the results above actually hold for a general Ricci shrinker without K\"{a}hler assumption. The special property which a K\"{a}hler-Ricci soliton possesses is that there is a natural toric action on $X$ associated with the soliton structure (cf.\cite{Bry08},\cite{CDS24},\cite{SZ25}), as shown in the following lemma.

\begin{lem}\label{lem:ToricActionOnX}
    The soliton vector field $\xi$ induces a holomorphic, isometric toric action.
\end{lem}

\begin{proof}
    Let $\Aut^0(X)$ be the identity component of the automorphism group of $X$ and $\Iso^0(X)$ be the identity component of the isometry group of $X$. It follows from \eqref{eq:KahlerRicciShrinker} that $\xi=J\nabla f$ is a real holomorphic and Killing vector field. Therefore, it generates a 1 parameter subgroup $\phi:\bR\to \Aut^0(X)\cap\Iso^0(X), \phi(t)=\exp(t\xi)$, where $\exp$ is the exponential map. Since $\Aut^0(X)\cap\Iso^0(X)$ is compact, the closure $\mathbb{T}=\overline{\{\exp(t\xi):t\in\bR\}}$ is a compact real torus, which acts holomorphically and isometrically on $X$. 
\end{proof}

\begin{rem}
    If we regard $(X,\omega)$ as a symplectic manifold, then the action of $\mathbb{T}$ is a Hamiltonian action. We shall not use this fact. For a proof of this, see \cite{CDS24}*{Proposition 5.14}.
\end{rem}

\section{Polarized Fano fibration structures on K\"{a}hler-Ricci shrinkers}

In this section we give a brief review of the notion of polarized Fano fibrations and the fact that any K\"{a}hler-Ricci shrinker can be realized as a polarized Fano fibration. We will use $\mathbb{T}$ to denote a compact real torus and let $\mathfrak{t}$ be its Lie algebra.

We start by recalling the definition of polarized affine cones, which was introduced in \cite{CS19} to study Sasaki-Einstein metrics:
\begin{defn}[Polarized affine cone]\label{defn:PolarizedAffineCone}
    A \textit{polarized affine cone} is a triple $(Y,\mathbb{T},\xi)$ such that
    \begin{enumerate}
        \item $Y$ is a normal affine cone in $\bC^m$ cut out by homogeneous polynomials;
        \item $\mathbb{T}$ acts on $Y$ effectively and the action extends to a linear action of $\mathbb{T}$ on $\bC^m$;
        \item Let $R(Y)$ be the ring of regular functions on $Y$ and 
        \begin{equation*}
            R(Y)=\oplus_{\alpha\in\mathfrak{t}^*}R(Y)_\alpha
        \end{equation*}
        be the weight decomposition of the action of $\mathbb{T}$ on $R(Y)$. $\xi\in\mathfrak{t}$ and $\langle\alpha,\xi\rangle>0$ for all $\alpha$ such that $R_\alpha\ne 0$. Such an element is called a \textit{Reeb vector field}.
    \end{enumerate}  
\end{defn}

For later use, recall the following weight space decomposition of the linear representation of tori:
\begin{lem}\label{lem:LinearRepresentationOfTorus}
    Suppose that $\rho:\mathfrak{t}\to GL(m,\bC)$ is the linear representation of $\mathfrak{t}$ on $\bC^m$ induced by the linear action of $\mathbb{T}$ on $\bC^m$. There is a basis $v_i,i=1,\cdots,m$ and $\alpha_i\in\mathfrak{t}^*$ ($\alpha_i$'s can be the same), such that 
    \begin{equation*}
        \rho(\eta)(v_i)=\sqrt{-1}\langle\alpha_i,\eta\rangle v_i, \,\ \forall \eta\in\mathfrak{t}.
    \end{equation*}
    Moreover, for $\alpha\in\mathfrak{t}^*$ in (3) of \ref{defn:PolarizedAffineCone}, $R_\alpha\ne 0$ if and only if 
    \begin{equation*}
        \alpha=\sum_i c_i\alpha_i,\,\  c_i\in \bZ_{\ge 0}.
    \end{equation*}
\end{lem}

By the lemma and a linear change of coordinates, we may assume that $v_i=e_i,i=1,\cdots,m$ is the standard basis of $\bC^m$. Therefore, let $z_i$ be any coordinate function on $\bC^m$ and $\eta\in\mathfrak{t}$, we obtain:
\begin{equation}
    \exp(\eta)^*z_i=e^{\sqrt{-1}\langle\alpha_i,\eta\rangle}z_i.
\end{equation}

We may complexify and get a linear action of $\mathbb{T}\otimes\bC$ on $\bC^m$. let $J$ be the complex structure on the complexified Lie algebra $\mathfrak{t}\otimes\bC$. Then for any $\eta\in\mathfrak{t}$,
\begin{equation}
    \exp(J\eta)^*z_i=e^{-\langle\alpha_i,\eta\rangle}z_i.
\end{equation}

In particular, for the Reeb vector field $\xi$, we have 
\begin{equation}\label{eq:DistanceOnYUnderFlow}
    \exp(-tJ\xi)^*|z|^2=\sum_{i=1}^m\exp(-tJ\xi)^*|z_i|^2=\sum_{i=1}^m e^{2\langle\alpha_i,\xi\rangle t}|z_i|^2.
\end{equation}
Since $\langle\alpha_i,\xi\rangle>0$, \eqref{eq:DistanceOnYUnderFlow} indicates that a point on $Y$ drifts away from the origin exponentially under the flow generated by $-J\xi$. This is crucial to the proof the main theorem.

The polarized affine cone was later generalized by Sun-Zhang (cf.\cite{SZ25}) to study stability and Yau-Tian-Donaldson type conjectures under a more general frame. They introduced the polarized Fano fibration:

\begin{defn}[Polarized Fano fibration]\label{defn:PolarizedFanoFibration}
    A \textit{polarized Fano fibration} consists of the data $(\pi:X\to Y,\mathbb{T},\xi)$ which satisfies the following conditions:
    \begin{enumerate}
        \item $(Y,\mathbb{T},\xi)$ is a polarized affine cone;
        \item $X$ is a nonsingular complex variety and $\mathbb{T}$ acts on $X$ effectively;
        \item $\pi:X\to Y$ is a $\mathbb{T}$ equivariant morphism. Moreover, it is also a \textit{Fano fibration}: If $\mathcal{O}_X$ and $\mathcal{O}_Y$ are structure sheaves of $X$ and $Y$, then $\pi_*\mathcal{O}_X=\mathcal{O}_Y$. In addition, the anticanonical bundle $K_X^{-1}$ is relatively ample over $Y$ in the sense that for all large enough $k$, there are finitely many sections $s_0,\cdots,s_N$ of $K_X^{-k}$ such that
        \begin{equation*}
            X\hookrightarrow Y\times\bC P^{N}, x\mapsto (\pi(x),[s_0(x):\cdots:s_N(x)]).
        \end{equation*}
        is an embedding.
    \end{enumerate}
\end{defn}

\begin{rem}
    Sun-Zhang allowed $X$ to have klt singularites in (2). We shall not need this general definition.
\end{rem}

The theorem below on K\"{a}hler-Ricci shrinker $(X,g,J,f)$ was proved in \cite{SZ25} using algebro-geometric methods developed in \cite{Bir21} and we refer readers to Section 3 of \cite{SZ25} for a proof:

\begin{thm}[\cite{SZ25}*{Theorem 3.1}]\label{thm:PolarizedFanoFibrationOnKahlerRicciShrinker}
    Let $(X,g,J,f)$ be a K\"{a}hler-Ricci shrinker and $\mathbb{T}$ be the compact real torus, which is the closure of the 1 parameter generated by the soliton vector field $\xi$ as in \ref{lem:ToricActionOnX}. There is an associated polarized Fano fibration structure $(\pi:X\to Y,\mathbb{T},\xi)$.
\end{thm}

\begin{emp}
    The following are the polarized Fano fibration structures on some well known K\"{a}hler-Ricci shrinkers:
    \begin{enumerate}
        \item Compact K\"{a}hler-Ricci shrinker $X$: $Y$ is a point. The Fano fibration is the constant map;
        \item The Gaussian shrinker $(\bC^n,g,\frac{1}{2}|z|^2)$: the toric action on $\bC^n$ is the $\mathbb{S}^1$ action generated by the soliton vector field $\xi=\sum_{i=1}^nz_i\frac{\partial}{\partial z_i}$. The Fano fibration $\pi:\bC^n\to\bC^n$ is the identity map;
        \item The product shrinker $\bC\times\bC P^1$: the soliton vector field generates $\mathbb{T}=\mathbb{S}^1$ action on $\bC\times\bC P^1$ by $\zeta\cdot(z,p)=(\zeta z,p)$. The Fano fibration $\pi:X\to \bC$ is the projection onto the first factor;
        \item The FIK shrinker $\mathcal{O}_{\bC{P}^{n-1}}(-k),0<k<n$ (cf.\cite{FIK03}): The Fano fibration is $\pi:\mathcal{O}_{\bC P^{n-1}}(-k)\to \bC^n/\bZ_k$, the blow down morphism which contracts the divisor with self intersection $-k$. The $\mathbb{S}^1$ action on $\bC^n/\bZ_k$ is induced from the standard $\mathbb{S}^1$ action on $\bC^n$: $\zeta\cdot(z_1,\cdots,z_n)=(\zeta z_1,\cdots,\zeta z_n)$. It lifts to the $\mathbb{S}^1$ action generated by the soliton vector field $\xi$ on $\mathcal{O}_{\bC{P}^{n-1}}(-k)$. Hence, $\pi$ is $\mathbb{S}^1$ equivariant.
    \end{enumerate}
\end{emp}

\section{Proof of the main theorem}

Let $(\pi:X\to Y,\mathbb{T},\xi)$ be a polarized Fano fibration structure on a K\"{a}hler-Ricci shrinker $(X,g,J,f)$ in \ref{thm:PolarizedFanoFibrationOnKahlerRicciShrinker}. Since $\pi$ is a projective morphism, the ring of holomorphic functions $\mathcal{O}(X)$ is isomorphic to $\mathcal{O}(Y)$. In this section, we further show that the ring of holomorphic functions with polynomial growth on $X$ are those on $Y$ and thus complete the proof. 

Define
\begin{equation}
    \mathcal{O}_d(Y)=\{f\in\mathcal{O}(Y):\exists C>0, |f(z)|\leq C(1+|z|^d)\}.
\end{equation}
We call an element $f$ in $\mathcal{O}_d(Y)$ a holomorphic function with polynomial growth of order at most $d$ on $Y$. Moreover, define the ring of holomorphic functions with polynomial growth on $Y$:
\begin{equation}
    \mathcal{O}_P(Y)=\cup_{d\ge 0}\mathcal{O}_d(Y).
\end{equation}  

The following extension theorem is a consequence of the $L^2$ estimate and is well known in the theory of several complex variable (cf.\cite{Bjo74}):

\begin{thm}\label{thm:PolynomialGrowthFunctionOnAffineVarieties}
    There is an integer $n(Y)$ such that if $f\in\mathcal{O}_d(Y)$, then there is a polynomial function $P$ on $\bC^m$ of degree at most $d+n(Y)$ such that $P=f$ on $Y$. In particular, $\mathcal{O}_P(Y)=R(Y)$. 
\end{thm}

\begin{rem}
    Note that the definition of $\mathcal{O}_P(Y)$ relies on a choice of the affine embedding $Y\hookrightarrow\bC^m$. The above theorem shows that $\mathcal{O}_P(Y)$ is indeed independent of the choice of the embedding, even though $\mathcal{O}_d(Y)$ does. 
\end{rem}

We are now able to prove \ref{thm:MainTheorem}. We firstly show the comparison between $f(x)$ and $|z(\pi(x))|$:

\begin{lem}
    For any $0<\beta<1-\gamma$, there are constants $C,c,K>0$ depending on the polarized Fano fibration structure on $X$ and $\beta$, such that for any point $x\in X$, the following inequalities hold:
    \begin{equation}\label{eq:EstimatesOnf}
        c|z(\pi(x))|^{\frac{2\beta}{\Lambda}}-K\leq f(x)\leq C|z(\pi(x))|^{\frac{2}{\lambda}}+K.
    \end{equation}
    Here 
    \begin{equation*}
        \lambda=\min_i\langle\alpha_i,\xi\rangle, \Lambda=\max_i\langle\alpha_i,\xi\rangle.
    \end{equation*}
\end{lem}
\begin{proof}
    Let $B_{r_\beta}(p)$ be the ball in \ref{lem:GradientEstimate} and $V=\pi^{-1}(\pi(\overline{B_{r_\beta}(p)}))$. Then $V$ is compact since $\pi$ is a proper morphism. It suffices to establish \eqref{eq:EstimatesOnf} on $X\setminus V$.
    
     To that end, let $B\supset V$ be an open ball in $Y$ and $W=\pi^{-1}(\overline{B}\setminus V)$. $W$ is also a compact subset of $X$. Moreover, let $\phi_t=\exp(-J\xi t)$ be the 1 parameter subgroup generated by $\nabla f= -J\xi$ and note that $\phi_t(W)\supseteq X\setminus V$ since $f$ has no critical points outside $V$. 
    
    Let $x\in W$, by (2) of \ref{prop:FundamentalResultsOnShrinkers},
    \begin{equation*}
        \frac{d}{dt}f(\phi_t(x))=|\nabla f(\phi_t(x))|^2\leq f(\phi_t(x)).
    \end{equation*}
    It follows that 
    \begin{equation}\label{eq:EstimateOnf}
        f(\phi_t(x))\leq f(x)e^{t}\leq e^t\max_W f, \forall t\ge 0.
    \end{equation}
    On the other hand, \eqref{eq:DistanceOnYUnderFlow} implies 
    \begin{equation}\label{eq:EstimateOnz}
        |z(\pi(\phi_t(x)))|^2=(\exp(-tJ\xi)^*|z|^2)(\pi(x))\ge e^{2\lambda t}\min_W|z\circ\pi|^2, \forall t\ge 0.
    \end{equation}
    Combine \eqref{eq:EstimateOnf} and \eqref{eq:EstimateOnz}, which implies that
    \begin{equation}
        f(\phi_t(x))\leq \left(\frac{\max_Wf}{\min_W |z\circ\pi|^{\frac{2}{\lambda}}}\right)|z(\pi(\phi_t(x)))|^{\frac{2}{\lambda}}, \forall t\ge 0.
    \end{equation}
    Therefore, we get
    \begin{equation}\label{eq:UpperBoundOnf}
        f(x)\leq C'|z(\pi(x)|^{\frac{2}{\lambda}}, \forall x\in X\setminus V\Rightarrow f(x)\leq C|z(\pi(x)|^{\frac{2}{\lambda}}+K, \forall x\in X.
    \end{equation}

    For the other direction, we need to make use of the assumption on the scalar curvature \eqref{eq:ScalarCurvatureBound}. By \ref{lem:GradientEstimate}, the following holds on $X\setminus V$:
    \begin{equation*}
        |\nabla f|^2\ge \beta f.
    \end{equation*}
    In particular, let $x\in W \subseteq X\setminus V$, we have 
    \begin{equation*}
        \frac{d}{dt}f(\phi_t(x))=|\nabla f(\phi_t(x))|^2\ge \beta f(\phi_t(x)).
    \end{equation*}
    This implies 
    \begin{equation}
        f(\phi_t(x))\ge e^{\beta t}f(x) \ge e^{\beta t}\min_Wf.
    \end{equation}
    Combine with 
    \begin{equation}
        |z(\pi(\phi_t(x)))|^2=(\exp(-tJ\xi)^*|z|^2)(\pi(x))\leq e^{2\Lambda t}\max_W|z\circ\pi|^2
    \end{equation}
    and we get
    \begin{equation}
        f(\phi_t(x))\ge\left(\frac{\min_W f}{\max_W|z\circ\pi|^{\frac{2\beta}{\Lambda}}}\right)|z(\pi(\phi_t(x)))|^{\frac{2\beta}{\Lambda}}.
    \end{equation}
    Thus the lower bound
    \begin{equation}
        f(x)\ge c|z(\pi(x))|^{\frac{2\beta}{\Lambda}}-K,\forall x\in X
    \end{equation}
    follows.
\end{proof}

Together with the estimate \eqref{eq:BoundOnfByDistance} on $f(x)$ by $r(x)$, this yields the comparison between $r(x)$ and $|z(\pi(x))|$: there are $C',c',K'>0$ such that
\begin{equation}\label{eq:CompareTheDistances}
    c'|z(\pi(x))|^{\frac{2\beta}{\Lambda}}-K'\leq r(x)^2 \leq C'|z(\pi(x))|^{\frac{2}{\lambda}}+K'.
\end{equation}

This shows that the distance function $r(x)$ on $X$ and the ``distance'' $|z|$ on $Y$ are comparable. From this, we may finish the proof of the main theorem

\begin{proof}[Proof of Main Theorem]
    Since $\pi$ has compact fiber, by maximum principle, holomorphic functions on $X$ are pullbacks of those on $Y$. Therefore, we may identify $\mathcal{O}(X)$ with $\mathcal{O}(Y)$. It follows from \eqref{eq:CompareTheDistances} that 
    \begin{equation*}
        \mathcal{O}_d(X)\subseteq\mathcal{O}_{\lceil\frac{d}{\lambda}\rceil}(Y).
    \end{equation*}
    Hence $\mathcal{O}_P(X)\subseteq\mathcal{O}_P(Y)$. Similarly, the lower estimate on $r$ in \eqref{eq:CompareTheDistances} yields the other direction: $\mathcal{O}_P(Y)\subseteq\mathcal{O}_P(X)$. Therefore, by \ref{thm:PolynomialGrowthFunctionOnAffineVarieties}, $\mathcal{O}_P(X)=\mathcal{O}_P(Y)=R(Y)$. In particular, $\mathcal{O}_P(X)$ is finitely generated. 
\end{proof}

\begin{rem}
    The curvature assumption \eqref{eq:ScalarCurvatureBound} is necessary for our argument. Without the assumption, it can only be shown that $\mathcal{O}_P(X)\subseteq\mathcal{O}_P(Y)=R(Y)$ is a subring of the finitely generated ring $R(Y)$. This doesn't imply the finite generation of $\mathcal{O}_P(X)$. For example, $\bC[xy,xy^2,\cdots]\subset\bC[x,y]$ is not finitely generated.
\end{rem}

\bibliography{Reference}

\end{document}